\documentclass{amsart}
\usepackage[utf8]{inputenc}

\usepackage{tikz}
\usepackage{xcolor}

\usepackage{amsmath}
\usepackage{amsthm}
\usepackage{amssymb}

\usepackage{booktabs}

\usepackage{hyperref}

\newtheorem{thm}{Theorem}[section]
\newtheorem{cor}[thm]{Corollary}

\newtheorem{prop}[thm]{Proposition}

\newtheorem*{ack}{Acknowledgements}

\theoremstyle{definition}
\newtheorem{defn}[thm]{Definition}

\newtheorem{exa}[thm]{Example}

\newcommand{\bbR}{\mathbb{R}}

\newcommand{\bbZ}{\mathbb{Z}}

\newcommand{\bbP}{\mathbb{P}}

\DeclareMathOperator{\conv}{Conv}
\DeclareMathOperator{\Int}{Int}
\DeclareMathOperator{\Star}{Star}

\title[Optimal transport stability of Weyl polytopes]{SYZ and Optimal transport stability of Weyl polytopes}

\author{Thibaut Delcroix}
\address{Univ Montpellier,
CNRS, Montpellier, France}
\email{thibaut.delcroix@umontpellier.fr}

\author{Jakob Hultgren}
\address{Dept of Mathematics and Mathematical Statistics, Umeå University, 901
87 Umeå, Sweden}
\email{jakob.hultgren@umu.se}

\date{\today}

\begin{document}

\begin{abstract}
We prove optimal transport stability (in the sense of Andreasson and the second author) for reflexive Weyl polytopes: reflexive polytopes which are convex hulls of an orbit of a Weyl group. 
When the reflexive Weyl polytope is Delzant, it follows from work of Li, Andreasson, Hultgren, Jonsson, Mazzon, McCleerey that the weak metric SYZ conjecture holds for the Dwork family in the corresponding toric Fano manifold. 
In particular, we show that the weak metric SYZ conjecture holds for centrally symmetric smooth Fano toric manifolds. 
\end{abstract}

\maketitle

\section{Introduction}

In this note, we observe that the optimal transport plan, with respect to the ambient duality bracket, between two measures invariant under a pair of dual oriented hyperplane reflections, must be supported on the union of corresponding positive half-spaces. 
Our goal is then to apply this simple observation to prove the weak metric SYZ conjecture for certain hypersurfaces in toric Fano manifolds which admit a large group of torus equivariant automorphism, by applying the results of Yang Li \cite{Li_2023_Duke,Li_2022_Acta,Li_2024_CBJ}, Andreasson, Hultgren, Jonsson, Mazzon, McCleerey \cite{HJMM_2024_Adv,Andreasson-Hultgren,AHJMM_2024}. 

This group of torus equivariant automorphisms acts on the associated character lattice \(M\), and on the associated vector space \(M_{\bbR}\). 
A reflection point group (a group generated by linear reflections) which preserves a lattice is called cristallographic, and is the Weyl group of a root system \cite[Chapitre~VI, \textsection~2, n°5]{Bourbaki}, as follows from the classification of discrete reflection groups by Coxeter \cite{Coxeter_1934}. 
We consider reflexive polytopes which are vertex-transitive under such a reflection point group (necessarily cristallographic). Such polytopes are called Weyl Polytopes and our main result establishes optimal transport stability for such polytopes (see \cite{Andreasson-Hultgren}). 
%Such polytopes are called reflexive Weyl polytopes, they are obtained as the convex hull of the orbit of a point under the action of the Weyl group of a root system. 

In addition, the arguments give enough control of the optimal transport map that the arguments for regularity in \cite[Theorem 3]{Andreasson-Hultgren} and \cite{AHJMM_2024} can be applied. Consequently, the solution is smooth and satisfies a global $C^{1,\alpha}$-bound away from a co-dimension 2 set.

\begin{thm}
A reflexive Weyl polytope is optimal transport stable.  Consequently, it admits (along with its dual) a solution to the real Monge-Ampère equation on its boundary. Moreover the solution on the Weyl polytope is smooth and satisfy a global $C^{1,\alpha}$-bound away from a piecewise affine set of co-dimension 2 and the solution on its dual is smooth and satisfy a global $C^{1,\alpha}$-bound on the open faces. 
\end{thm}

Although reflexive polytopes in general have few symmetries, we highlight several infinite families of examples of Weyl polytopes from \cite{Voskresenskij-Klyachko_1985,Montagard-Rittatore_2023}. For illustration, consider the Gorenstein toric Fano variety \(X_d\) given by the face fans of the polytope  
$$ \Delta_d^{\vee} = \conv\{ \pm e_0,\ldots, \pm e_d, \pm\sum_{j=0}^d e_j\} $$
for $d\geq 2$. 
Then the moment polytope \(\Delta_d\) of \(X_d\) (which is the polytope dual to \(\Delta_d^{\vee}\)) is a Weyl polytope for the root system \(A_d\). 
If $d$ is even, then \(\Delta_d\) is Delzant or, equivalently, \(X_d\) is smooth. 
Recall that a toric variety \(X\) is centrally symmetric if there exists a torus-equivariant automorphism \(\sigma\) of \(X\) which restricts to the inversion on the open-dense torus orbit. 
By \cite[Theorem~6]{Voskresenskij-Klyachko_1985}, a smooth and Fano centrally symmetric toric variety is a product of varieties of the form \(X_{d}\) with \(d\) even, or \(\bbP^1\). 

Any toric Fano variety $X$ has a special family of Calabi-Yau hypersurfaces given by 
$$ Y_t = \{f_0 + t\sum f_m = 0\}$$
where $d+1$ is the dimension of $X$, $f_0$ is the unique $(\mathbb C^*)^{d+1}$ invariant anti-canonical section and the sum is taken over the anti-canonical sections corresponding to the vertices of the moment polytope of $X$. Extending terminology from the case $Y=\mathbb P^{d+1}$, we will refer to this family as the \emph{Dwork family} in $Y$.  
As a consequence of our main result, we obtain for example the following: 

\begin{cor}
The weak metric SYZ conjecture holds for the Dwork family in centrally symmetric smooth Fano toric manifolds, in other words, for each $\delta>0$, there is $t_0>0$ such that if $t<t_0$, then $Y_t$ admits a special Lagrangian torus fibration on a subset of relative volume $(1-\delta)$. 
\end{cor}

In addition to this, we determine precisely which three dimensional reflexive polytopes are Weyl polytopes. 
In conjunction with Yang Li's sufficient condition (see \cite{Li_2024_CBJ}), we deduce that, among three dimensional reflexive polytopes with automorphism group of order strictly larger than \(8\), only three may not satisfy optimal transport stability: polytopes number 3036, 735 (wich are dual to each other) and 2355 (which, incidentally, is not Kähler-Einstein). 
The numbers here and in the body of the paper refer to the identifier of 
a three-dimensional reflexive polytope as classified by Kreuzer and Skarke \cite{Kreuzer-Skarke_1997,Kreuzer-Skarke_1998}, as encoded in Kasprzyk's grdb database \cite{Kasprzyk_2010,grdb} of canonical Fano polytopes (reflexive polytopes are canonical). 
Note that the numbering starts at \(1\) instead of \(0\) so that polytope number \(1\) is the (dual of the moment) polytope of \(\bbP^3\). 

\begin{ack}
The first author is partially funded by ANR-21-CE40-0011 JCJC project MARGE. The second author is partially funded by the Swedish Research Council Starting Grant 2023-05485. 
The first author thanks Pierre-Louis Montagard for various discussions related to \cite{Montagard-Ressayre_2009} and \cite{Montagard-Rittatore_2023}. The second author thanks Rolf Andreasson, Mattias Jonsson, Yang Li, Enrica Mazzon and Nick McCleerey for many discussions on the subject. 
\end{ack}

\section{Reflexive Weyl polytopes}

\subsection{Reflexive polytopes}

We fix \(M\simeq \bbZ^{d+1}\) a lattice, we denote by \(N = \hom(M,\bbZ)\) the dual lattice, and we let \(M_{\bbR}=M\otimes \bbR\) and \(N_{\bbR}=N\otimes \bbR\) be the associated real vector spaces. 
We denote by \(\langle \cdot, \cdot \rangle : M_{\bbR}\times N_{\bbR} \to \bbR\) the duality bracket.

A \emph{lattice polytope} \(\Delta\subset M_{\bbR}\) is the convex hull of a finite set of elements of \(M\). 
We denote by \(V(\Delta)\) the set of vertices of \(\Delta\). 
If \(0\in \Int(\Delta)\), the dual polytope \(\Delta^{\vee}\) is the convex polytope in \(N_{\bbR}\) defined by 
\[ \Delta^{\vee} = \{ n\in N_{\bbR} \mid \forall m\in V(\Delta), \langle m,n\rangle \leq 1\} \]
Given \(m\in V(\Delta)\), we denote by 
\[ \tau_m = \{ n\in \Delta^{\vee} \mid \langle m,n\rangle =1 \} \]
the facet of \(\Delta^{\vee}\) defined by \(m\). 
We also let \(\Star(m)\) be the \emph{closed star} of \(m\), that is, the union of all closed faces of \(\partial \Delta\) containing \(m\). 

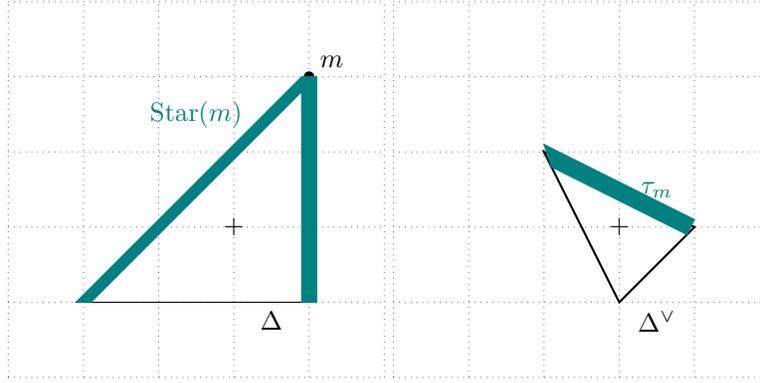
\begin{center}
\begin{tikzpicture}
\draw[dotted,very thin] (-3,3) grid (2,-2);
\draw (0,0) node{+};
\draw (1,-1) -- (1,2) -- (-2,-1) -- cycle;
\draw (1,2) node{\(\bullet\)};
\draw (1.3,2.2) node{\(m\)};
\draw (0.5,-1) node[below]{\(\Delta\)};
\draw[teal] (-0.5,1.5) node{\(\Star(m)\)};
\draw [draw=teal, fill=teal, opacity=0.2] (-2.1,-1) -- (0.9,2) -- (1.1,2) -- (1.1,-1) -- (0.9,-1) -- (0.9,1.8) -- (-1.9,-1) -- cycle;
\end{tikzpicture}
\begin{tikzpicture}%[scale=0.6]
\draw[dotted,very thin] (-3,3) grid (2,-2);
\draw (0,0) node{+};
\draw[thick] (-1,1) -- (0,-1) -- (1,0) -- cycle;
\draw (0.5,-1) node[below]{\(\Delta^{\vee}\)};
%\draw[dashed,ultra thick,teal] (-1,1) -- (1,0);
\draw [draw=teal, fill=teal, opacity=0.2] (-1,1.1) -- (1,.1) -- (0.9,-0.1) -- (-0.9,0.8) --cycle;
\draw[teal] (0.5,0.5) node{\(\tau_m\)};
\end{tikzpicture}
\end{center}

\begin{defn}
A lattice polytope \(\Delta\subset M_{\bbR}\) is \emph{reflexive} if it contains the origin in its interior and its dual \(\Delta^{\vee}\) is also a lattice polytope. 
\end{defn}

Finally, let us highlight the sufficient condition of optimal transport stability obtained by Yang Li \cite{Li_2024_CBJ} for future reference.

\begin{defn}
We say that a reflexive polytope \(\Delta\) satisfies 
the \emph{vertex condition}
if for every pair \((m,n)\) where \(m\) is a vertex of \(\Delta\) and \(n\) is a vertex of \(\Delta^{\vee}\), one has 
\[ \langle m, n \rangle \neq 0 \]
\end{defn}

\subsection{Weyl polytopes}
\label{sec_Weyl_polytopes}

We recall basic definitions and properties of root systems, using \cite[Chapitre VI]{Bourbaki} as a reference, and introduce Weyl polytopes. 

\begin{defn}
A set \(\Phi\)  is a root system in \(M_{\bbR}\) if 
\begin{enumerate}
    \item \(0\notin \Phi \subset M_{\bbR}\), 
    \item \(\Phi\) spans \(M_{\bbR}\) as a vector space,
    \item for every \(\alpha\in \Phi\), there exists an \(\alpha^{\vee}\in N_{\bbR}\) such that \(\langle \alpha,\alpha^{\vee}\rangle = 2\), the reflection \(\sigma_{\alpha}:m\mapsto m-\langle m,\alpha^{\vee}\rangle \alpha\) sends \(\Phi\) to \(\Phi\), and for any \(\alpha\in \Phi\), \(\alpha^{\vee}(\Phi)\subset \bbZ\). 
\end{enumerate}
\end{defn}
We denote the group generated by the \(\sigma_{\alpha}\) by \(W=W(\Phi)\), it is called the \emph{Weyl group} of the root system. 
The \(\alpha^{\vee}\) are uniquely determined, they form the \emph{dual root system} \(\Phi^{\vee}\subset N_{\bbR}\), and \((\alpha^{\vee})^{\vee}=\alpha\). 
The Weyl group \(W^{\vee}\) of the root system \(\Phi^{\vee}\) is isomorphic to \(W\) via \(\sigma_{\alpha}\mapsto \sigma_{\alpha^{\vee}}\), and this correspondence preserves the duality bracket: for any \(m\in M_{\bbR}\) and \(n\in N_{\bbR}\), we have \(\langle \sigma_{\alpha}(m),n\rangle = \langle m,\sigma_{\alpha^{\vee}}(n)\rangle\).

Let \(S=\{\alpha_1,\cdots,\alpha_{d+1}\} \subset \Phi\) denote a set of \emph{simple roots}, that is, such that each \(\alpha\in \Phi\) writes as \(\sum_{i=1}^{d+1} x_i \alpha_i\) with either all \(0\leq x_i\in \bbZ\) or all \(0\geq x_i \in \bbZ\). 
It always exists, and \(W\) acts simply transitively on the set of sets of simple roots. 
Given a fixed choice of \(S\), we define:

\begin{defn}
The \emph{positive Weyl chamber} \(C^+_M\) in \(M_{\bbR}\) is the set 
\[ C^+_M = \{m\in M_{\bbR} \mid \forall \alpha\in S, \langle m, \alpha^{\vee}\rangle \geq 0 \}. \]
The \emph{positive dual Weyl chamber} \(C^+_N\) in \(N_{\bbR}\) is the set 
\[ C^+_N = \{n\in N_{\bbR} \mid \forall \alpha\in S, \langle \alpha, n \rangle \geq 0 \}.\]
\end{defn}

\begin{exa}
The figure below represents the root system \(B_2\). The roots are the ends of the arrows, the two positive roots \(\alpha_1\) and \(\alpha_2\) are indicated, and the positive Weyl chamber is indicated in dashed lines. 
The Weyl group of this root system is the dihedral group of order \(8\) which is the group of isometries of a square. 
\begin{center}
\begin{tikzpicture}
%\draw[dotted,very thin] (-3,3) grid (2,-2);
%\draw (0,0) node{+};
\draw[<->, thick] (-1,-1) -- (1,1);
\draw[<->, thick] (1,-1) -- (-1,1) node[left]{\(\alpha_1\)};
\draw[<->, thick] (0,-1) -- (0,1);
\draw[<->, thick] (-1,0) -- (1,0) node[right]{\(\alpha_2\)};
\draw[dashed, very thin] (0,0) -- (2,2); 
\draw[dashed, very thin] (0,0) -- (0,2); 
\end{tikzpicture}
\end{center}
\end{exa}

\begin{prop}{\cite[Chapitre VI, \textsection 1, n°5]{Bourbaki}}
The positive Weyl chamber is a fundamental domain for the action of \(W\) on \(M_{\bbR}\). 
\end{prop}

\begin{defn}
Given \(m\in M_{\bbR}\), the polytope 
\[ \Delta_W(m) := \conv \{w\cdot m \mid w\in W \} \]
is called the \emph{Weyl polytope} associated with \(m\). 
We denote by \(\Delta_W^{\vee}(m)\) the polytope dual to \(\Delta_W(m)\)
\end{defn}

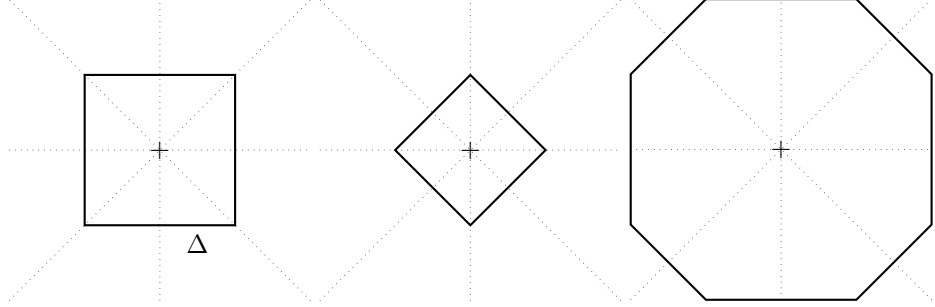
\begin{figure}
\caption{Weyl polytopes for the root system \(B_2\)}
\label{fig_Weyl_B2}
\begin{center}
\begin{tikzpicture}
%\draw[dotted,very thin] (-2,2) grid (2,-2);
\draw (0,0) node{+};
\draw[dotted, very thin] (-2,-2) -- (2,2); 
\draw[dotted, very thin] (0,-2) -- (0,2); 
\draw[dotted, very thin] (2,-2) -- (-2,2); 
\draw[dotted, very thin] (-2,0) -- (2,0); 
\draw[thick] (1,-1) -- (1,1) -- (-1,1) -- (-1,-1) -- cycle;
\draw (0.5,-1) node[below]{\(\Delta\)};
\end{tikzpicture}
\begin{tikzpicture}%[scale=0.6]
%\draw[dotted,very thin] (-2,2) grid (2,-2);
\draw (0,0) node{+};
\draw[thick] (-1,0) -- (0,-1) -- (1,0) -- (0,1) -- cycle;
\draw[dotted, very thin] (-2,-2) -- (2,2); 
\draw[dotted, very thin] (0,-2) -- (0,2); 
\draw[dotted, very thin] (2,-2) -- (-2,2); 
\draw[dotted, very thin] (-2,0) -- (2,0); 
\end{tikzpicture}
\begin{tikzpicture}%[scale=0.6]
%\draw[dotted,very thin] (-2,2) grid (2,-2);
\draw (0,0) node{+};
\draw[thick] (-1,-2) -- (1,-2) -- (2,-1) -- (2,1) -- (1,2) -- (-1,2) -- (-2,1) -- (-2,-1) -- cycle;
\draw[dotted, very thin] (-2,-2) -- (2,2); 
\draw[dotted, very thin] (0,-2) -- (0,2); 
\draw[dotted, very thin] (2,-2) -- (-2,2); 
\draw[dotted, very thin] (-2,0) -- (2,0); 
\end{tikzpicture}
\end{center}
\end{figure}

Examples of Weyl polytopes for the root system \(B_2\) are given in Figure~\ref{fig_Weyl_B2}. 
Note that, since \(C^+_M\) is a fundamental domain for the action of \(W\) on \(M_{\bbR}\), any Weyl polytope has a vertex in \(C^+_M\), and we may always assume that \(m\in C_M^+\).  
However, this vertex can (and will often in our examples) be on the boundary of \(C^+_M\). 
Note also that, although a Weyl polytope always contains the origin in its relative interior, in general, the dual of a Weyl polytope is not a Weyl polytope. 

The \emph{weight lattice} \(\Lambda_w\subset M_{\bbR}\) is the set of elements \(m\in M_{\bbR}\) such that for any \(\alpha\in\Phi\), \(\langle m, \alpha^{\vee} \rangle \in \bbZ\).  
The \emph{fundamental weights} are the elements of the dual basis \((\omega_1,\ldots, \omega_{d+1})\) of the basis of simple coroots \((\alpha_1^{\vee},\ldots, \alpha_{d+1}^{\vee})\). 
They generate the lattice \(\Lambda_w\). 
The \emph{root lattice} \(\Lambda_r\subset M_{\bbR}\) is the lattice generated by \(\Phi\). 
We fix a root system \(\Phi\) in \(M_{\bbR}\), with the additional condition that \(\Lambda_r\subset M \subset \Lambda_w\). 

We will be interested in Weyl polytopes which are reflexive, and we will exhibit several families of examples.  

\begin{prop}
\label{prop:Cones}
If \(m\in M\cap C^+_M\), then \(\Delta_W(m)\) is a \(W\)-invariant, vertex-transitive lattice polytope which contains the origin in its interior. 
Furthermore, 
\[ \partial\Delta_W(m) \cap C^+_M \subset \Star(m) \] 

The dual \(\Delta_W^{\vee}(m)\) is invariant under the action of \(W^{\vee}\), and the action induced by \(W^{\vee}\) on facets of \(\Delta_W^{\vee}(m)\) is transitive. 
Furthermore, 
\[\partial\Delta_W^{\vee}(m) \cap C^+_N \subset \tau_m \] 
\end{prop}

\begin{proof}
Most properties follow readily from the definition, and the fact that \(C_M^+\) is a fundamental domain for the action of \(W\). 
Let us expand on the properties \(\partial\Delta_W(m) \cap C^+_M \subset \Star(m)\) and \(\partial\Delta_W^{\vee}(m) \cap C^+_N\subset \tau_m\). 
For this, we rely on \cite[Section~4]{Montagard-Rittatore_2023} who provide a precise description of \(\Delta_W^{\vee}(m)\). 

Set \(L= \{i \mid \langle \alpha^{\vee}_i,m\rangle= 0\}\). 
Let \(W_L\) be the subgroup of \(W\) generated by the reflections \(\sigma_{\alpha_i}\), for \(i\in L\). 
Then the vertices of \(\tau_m\) are on the extremal rays of the cone 
\[ C_L = \bigcup_{w\in W_L} w\cdot C_N^+ \]
and each extremal ray does contain a vertex. 
Since \(C_L\) contains \(C_N^+\), we obtain the property \(\partial\Delta_W^{\vee}(m) \cap C^+_N\subset \tau_m\). 

As another consequence, by vertex transitivity on \(\Delta_W(m)\), any vertex of \(\Delta_W^{\vee}(m)\) is of the form \(w(\omega_i^{\vee})\) for some \(w\in W\) and some fundamental coweight \(\omega_i^{\vee}\). 
Furthermore, if \(n\in V(\Delta_W^{\vee}(m))\setminus V(\tau_m)\), then \(w\) is a non-trivial element of the Weyl group. 

Let us now prove that \(\partial\Delta_W(m) \cap C^+_M \subset \Star(m)\) by contradiction. 
Assume that there exists a \(x\in C_M^+\cap \partial \) such that \(x\notin \Star(m)\). 
Then by the above there exists a non-trivial \(w\in W\), an \(i\in \{1,\ldots,d+1\}\), and a constant \(c>0\) such that \(w(\omega_i^{\vee})(p)\leq c\) for \(p\in \Delta_W(m)\), and  
 \(w(\omega_i^{\vee})(x)= c\). 
 Since \(w\) is non-trivial, \(w(\omega_i^{\vee})\notin C_N^+\), so there exists \(\alpha\in S\) with \(w(\omega_i^{\vee})(\alpha)<0\). 
 Since \(\Delta_W(m)\) is \(W\)-stable, \(\sigma_{\alpha}(x)\in \Delta_W(m)\), and by convexity, \(tx+(1-t)\sigma_{\alpha}(x) = x - (1-t)\langle x,\alpha^{\vee}\rangle \alpha \in \Delta_W(m)\) for all \(t\in [0,1]\). 
 Now, 
 \begin{align*}
 w(\omega_i^{\vee})(tx+(1-t)\sigma_{\alpha}(x)) & 
 = w(\omega_i^{\vee})(x) -(1-t)\langle x,\alpha^{\vee}\rangle w(\omega_i^{\vee})(\alpha) \\ 
 & = c - (1-t)\langle x,\alpha^{\vee}\rangle w(\omega_i^{\vee})(\alpha) \\
 & > c
 \end{align*}
 for \(t<1\), since \(x\in C_M^+\) and \(w(\omega_i^{\vee})(\alpha)<0\). 
 This contradicts that \(w(\omega_i^{\vee})(p)\leq c\) for \(p\in \Delta_W(m)\). 
\end{proof}

\subsection{Infinite families of examples}

In \cite{Montagard-Rittatore_2023}, Montagard and Rittatore classified the reflexive Weyl polytopes when \(\Phi\) is an irreducible root system and \(M=\Lambda_r\), obtaining eleven infinite families of examples plus some exceptional cases.  
In particular, they recover the smooth examples, which were previously classified by Voskresenskij and Klyachko in \cite{Voskresenskij-Klyachko_1985}. 
Using Bourbaki's standard ordering of roots, the families are summarized in Table~\ref{table_MR}.

\begin{table}
\begin{center}
\caption{Montagard-Rittatore polytopes \(\Delta_W(m)\)}
\label{table_MR}
\begin{tabular}{ccccc}
\toprule
Type & dimension & \(m=\) & variety & smooth? \\
\midrule
\(A_n\) & \(n\geq 1\) & \((n+1)\omega_1\) & \(\bbP^n\) & Yes \\ 
\(A_n\) & \(n\geq 2\) & \(\omega_1+\omega_n\) & &  \\ 
\(A_{2k+1}\) & \(k\geq 1\) & \(2\omega_{k+1}\) & \(V_{2k+1}\) &  \\ 
\(A_{2k}\) & \(k\geq 2\) & \((2k+1)(\omega_{k}+\omega_{k+1})\) & \(V_{2k}\) & Yes \\ 
\(B_n\) & \(n\geq 2\) & \(\omega_1\) & & \\
\(B_n\) & \(n\geq 2\) & \(2\omega_n\) & \((\bbP^1)^n\) & Yes \\
\(C_n\) & \(n\geq 3\) & \(2\omega_1\) & &  \\ 
\(C_n\) & \(n\geq 3\) & \(\omega_2\) & &  \\ 
\(D_n\) & \(n\geq 4\) & \(2\omega_1\) & &  \\
\(D_n\) & \(n\geq 4\) & \(\omega_2\) & &  \\ 
\(E_6\) & \(6\) & \(\omega_2\) & &  \\
\(F_4\) & \(4\) & \(\omega_4\) & &  \\
\(G_2\) & \(2\) & \(\omega_2\) & \(V_2\) & Yes \\  
\bottomrule
\end{tabular}
\end{center}
\end{table}

\begin{exa}
The variety \(V_n\) (corresponding to types $A_n$ and $G_2$), which is smooth for even \(n\) and singular for odd \(n\), is a toric variety which is isomorphic in codimension \(1\) to the blow up of \((\bbP^1)^n\) at two points, and which may be described (alternatively from the description of its moment polytope \(\Delta\) as in the table), as the Gorenstein toric Fano variety such that \(\Delta^{\vee}\) is the convex hull  
\[ \Delta^{\vee}= \conv \{\pm e_1, \ldots, \pm e_n, \pm(e_1+\cdots+e_n)\} \]
where \(e_1,\ldots,e_n\) is a basis of \(N\). 
Note that all the points we take the convex hull of turn out to be vertices of the polytope. Note that $\Delta$, and not in general $\Delta^\vee$, is a Weyl polytope. 

Let us describe one vertex of \(\Delta\) in this description. For this, consider the dual basis \(e_1^*,\ldots,e_n^*\) of \(M\), and the element \(e_1^*+\cdots+e_{\lfloor n/2\rfloor}^*-e_{\lfloor n/2\rfloor+1}^*-\cdots e_n^*\). 
One checks readily that the duality bracket between this element and the vertices of \(\Delta^{\vee}\) is always less than \(1\), with equality exactly for the vertices 
\(e_1,\ldots, e_{\lfloor n/2\rfloor}, -e_{\lfloor n/2\rfloor+1},\ldots,-e_n\) if \(n\) is even, and the same vertices plus the vertex \(-e_1-\cdots -e_{\lfloor n/2\rfloor}+e_{\lfloor n/2\rfloor+1}+\cdots +e_n\) if \(n\) is odd. 
In case \(n\) is even, we have 
\[ \langle e_1^*+\cdots+e_{\lfloor n/2\rfloor}^*-e_{\lfloor n/2\rfloor+1}^*-\cdots e_n^*, e_1+\cdots+e_n \rangle = 0\]
so that the vertex condition is not satisfied. 
%From the above and vertex transitivity, on the other hand, condition~YL is satisfied for \(n\) odd. 
%We also recover that \(V_n\) is not smooth if \(n\) is odd (its facets are not even simplicial).
\end{exa}

\subsection{Examples in low dimensions} 

In dimension 1, there is only one reflexive polytope, the moment polytope of \(\bbP^1\), which is indeed a Weyl polytope for the root system \(A_1\). 

In dimension 2, among the \(16\) reflexive polytopes, five are reflexive Weyl polytopes. 
These are the moment polytope of \(\bbP^2\) and its dual, the moment polytope of \(\bbP^1\times \bbP^1\) and its dual, and the moment polytope of the blow up of \(\bbP^2\) at three points, which coincides with the variety \(V_2\) defined in the previous section. 
Note that they are all smooth.
As noted before, the polytope of \(V_2\) does not satisfy the vertex condition. 

In dimension 3, we determine the Weyl polytopes and their dual as follows, using  
the classification of three-dimensional reflexive polytopes by Kreuzer and Skarke \cite{Kreuzer-Skarke_1997,Kreuzer-Skarke_1998}, as enumerated in Kasprzyk's grdb database \cite{Kasprzyk_2010,grdb} of canonical Fano polytopes by \emph{reflexive ID}.  

\begin{prop}
    All 3-dimensional reflexive Weyl polytopes have automorphism group of order strictly larger than 8. 
Furthermore, there are only 5 reflexive 3-dimensional polytopes with automorphism group of order larger than 8 that are not Weyl polytopes or dual of Weyl polytopes: polytopes number 735, 3036, 2355, 156 and 4249. Among these 5, only the last 2 satisfy the vertex condition. 
\end{prop}

\begin{proof}
In dimension 3, the possible root systems that can be used to construct Weyl polytopes are \(A_3\), \(B_3\), \(C_3\), and products of lower rank root systems. 
In particular, their automorphism group is of order at least \(8\) (and if it is equal to \(8\), the root system is \((A_1)^3\)). 
Furthermore, by symmetry, the barycenter of a Weyl polytope is the origin. 
This narrows down the list of candidates to 28. 

It is not hard to check that the nine polytopes whose automorphism group have order eight are not Weyl polytopes. Indeed, polytopes 9 and 3314 are simplices hence they cannot be Weyl polytopes or dual of Weyl polytopes for \((A_1)^3\). Polytopes number 199, 3416, 2131, 610 are obviously not vertex transitive since for some but not all of their vertices, the opposite lattice point is also a vertex. As a consequence, they cannot be Weyl polytopes, and since their dual is in the same list, they cannot be dual to Weyl polytopes. Polytope 1324 is not facet transitive (some facets have four vertices and some have three), and self-dual, so it cannot be a Weyl polytope or its dual. 
Finally, polytope 4167 is dual to polytope 25, and one can check that their automorphism group is not abelian (which would be the case for the Weyl group of \((A_1)^3\), already of order \(8\)). 
Apart from the two simplices, none of these polytopes satisfy the vertex condition. 

The four polytopes with automorphism group of order \(12\) are Weyl polytopes or the dual of a Weyl polytope for the product root system \(A_1\times A_2\). 
The polytope number 4287 is the moment polytope of the smooth threefold \(\bbP^1\times \bbP^2\), it is a reflexive Weyl polytope, whose dual is the polytope number 5. 
The polytope number 776 is the product of the moment polytope of \(\bbP^1\) with the dual of the moment polytope of \(\bbP^2\). It is a reflexive Weyl polytope whose dual polytope is the polytope number 769. 

There are \(4\) polytopes with automorphism group of order \(16\). The polytopes number 3036 and 735 are dual to each other, and none of the two are facet-transitive, so none are Weyl polytopes. Furthermore, they do not satisfy the vertex condition. 
On the other hand, polytope number 2078 is the product of the moment polytope of \(\bbP^1\) with the polytope dual to the moment polytope of \(\bbP^1\times \bbP^1\), and it is a reflexive Weyl polytope for the root system \(A_1\times B_2\) (or, also, \((A_1)^3\)). 
Its dual is polytope number 510.

The five polytopes with automorphism group of order \(24\) all appear as Weyl polytopes or duals of Weyl polytopes. 
Polytope number 3875 is the moment polytope of \(\bbP^1\times V_2\), a Weyl polytope for the root system \(A_1\times G_2\) or \(A_1\times A_2\), and its dual is polytope number 219. 
Note that these polytopes do not satisfy the vertex condition.
Polytope number 4312 is the moment polytope of \(\bbP^3\), a Weyl polytope for the root system \(A_3\), and its dual is polytope number 1. 
The fifth polytope, numbered 428, is a self-dual reflexive polytope. It appears as a Weyl polytope in \cite{Montagard-Ressayre_2009} as the regular simplex \(\mathcal{S}_2^3\). 

Finally, the six polytopes with automorphism group of order \(48\) all appear as Weyl polytopes from \cite{Montagard-Rittatore_2023} or their duals. 
Polytope number 1530 is the Weyl polytope \(\Delta_W(\omega_1+\omega_3)\) for the root system \(A_3\), or the Weyl polytope \(\Delta_W(\omega_2)\) for the root system \(C_3\). 
Its dual is polytope number 2356. 
These two polytopes do not satisfy the vertex condition.  
Polytope number 3350 is the Weyl polytope \(\Delta_W(2\omega_2)\) for the root system \(A_3\), or the Weyl polytope \(\Delta_W(2\omega_1)\) for the root system \(C_3\). 
Its dual is polytope number 198. 
Finally, polytope number 31 is the moment polytope of \((\bbP^1)^3\). It is the Weyl polytope \(\Delta_W(\omega_1)\) for the root system \(B_3\) (or also a Weyl polytope for the root systems \((A_1)^3\) or \(A_1\times B_2\)). 
Its dual is polytope number 4251, which is as well a Weyl polytope: \(\Delta_W(2\omega_3)\) for the root system \(B_3\). 

To prove the second part of the statement, we observe that there are only three reflexive polytopes with automorphism goup of order strictly larger than 8 and barycenter different from the origin. 
These are polytopes number 156 and 4249, which satisfy the vertex condition, and polytope number 2355, which does not satisfy the vertex condition. 
\end{proof}

\section{Optimal transport stability of reflexive Weyl polytopes}

\subsection{Reflections and optimal transport}

We consider a (linear) reflection \(\sigma\) of \(M_{\bbR}\), that is, an isomorphism \(M_{\bbR}\to M_{\bbR}\) which is an involution, and leaves stable a linear hyperplane \(H\subset M_{\bbR}\). 
For consistency with the previous section, we introduce a root formalism. 
Let \(\alpha\) be an eigenvector for the eigenvalue \(-1\), which is well-defined up to multiple, and let \(\alpha^{\vee}\) be the element of \(N_{\bbR}\) such that \(\langle h, \alpha^{\vee}\rangle=0\) for \(h\in H\), and \(\langle \alpha,\alpha^{\vee}\rangle=2\). 
Then for \(m\in M_{\bbR}\), we can write 
\begin{equation} \label{eq:reflections} \sigma(m)=m-\langle m, \alpha^{\vee}\rangle\alpha \end{equation}
Furthermore, \(\sigma\) defines a corresponding reflection \(\sigma^{\vee}\) of \(N_{\bbR}\), obtained by exchanging the roles of \(\alpha\) and \(\alpha^{\vee}\), such that for \(n\in N_{\bbR}\), \(\sigma^{\vee}(n)=n-\langle \alpha, n\rangle\alpha^{\vee}\). 

We consider a (convex) polytope \(\Delta \subset M_{\bbR}\) containing the origin, and stable under \(\sigma\). 
Then the dual polytope \(\Delta^{\vee}\subset N_{\bbR}\) is stable under \(\sigma^{\vee}\).

We further consider finite measures \(\mu\) on \(\partial \Delta\) and \(\nu\) on \(\partial \Delta^{\vee}\) which are respectively \(\sigma\) and \(\sigma^{\vee}\) invariant. 
We consider a $(\sigma,\sigma^\vee)$-invariant optimal transport plan from \(\mu\) to \(\nu\), for the cost function 
\[
c:
\partial \Delta \times \partial \Delta^{\vee} \to \bbR, 
(m,n) \mapsto -\langle m, n\rangle 
\]
This means the transport plan is a $(\sigma,\sigma^\vee)$-invariant minimizer of 
$$ \gamma \mapsto \int_{\partial\Delta\times \partial\Delta^\vee} c\gamma $$
and, consequently, that the support of the transport plan is $c$-cyclically monotone, i.e. 
\[c(m_1,n_1)+\ldots\ + c(m_k,n_k) - (c(m_1,n_2)+\ldots+c(m_k-1,n_k)+c(m_k,n_1))\leq 0 \]
for all set of pairs $(m_1,n_1),\ldots,(m_k,n_k)$ in the support of the transport plan. 
%Note that this cost function is, in particular, continuous. 

\begin{prop} 
\label{prop_OT_and_reflection}
For any \((m,n)\) in the support of a $W$-invariant optimal transport plan, we have 
\[ \langle m,\alpha^{\vee}\rangle \langle \alpha, n \rangle  \geq 0 \]
\end{prop}

\begin{proof}
This is a direct consequence of \(c\)-cyclical monotonicity. 
Indeed, by uniqueness, the transport plan is invariant under the involution \((\sigma,\sigma^{\vee})\), so if \((m,n)\) is in the support of the optimal transport plan, then \((\sigma(m),\sigma(n))\) is as well. 
Applying \(c\)-cyclical monotonicity to these two points of \(M_{\bbR}\times N_{\bbR}\), we have 
\[ c(m,n)+c(\sigma(m),\sigma^{\vee}(n)) \leq c(m,\sigma^{\vee}(n))+c(\sigma(m),n) \]
or, by substituting $c(\cdot,\cdot)$ by $-\langle \cdot,\cdot\rangle$ and rearrangeing, 
\[ \langle m-\sigma(m),n-\sigma^{\vee}(n)\rangle \leq 0 \]
Going back to the explicit expressions \(\sigma\), we have 
\[ \langle m-\sigma(m),n-\sigma^{\vee}(n)\rangle = -\langle \langle m,\alpha^{\vee}\rangle \alpha, \langle \alpha, n \rangle \alpha^{\vee}\rangle = -2\langle m,\alpha^{\vee}\rangle \langle \alpha, n \rangle \]
hence the result. 
\end{proof}

\subsection{Optimal transport stability of reflexive Weyl polytopes}
Given a lattice polytope of dimension $d+1$, its integral surface measure is the measure supported on the open faces of its boundary defined by the following two facts:
\begin{itemize}
	\item it restricts to a multiple of $d$-dimensional Lebesgue measure on each open face 
	\item it assigns volume $1/d!$ to any simplex spanned by a set of generators of the sublattice contained in the affine subspace spanned by the open face.
\end{itemize}
We will use $\mu_M$ to denote the integral surface measure on $\partial\Delta$ and $\nu_N$ to denote the integral surface measure on $\partial\Delta^\vee$. A central point is given by the real Monge-Ampère equation
\begin{eqnarray}
    & & \Phi: M_{\mathbb R} \rightarrow \mathbb R \text{ convex} \label{eq:MA1} \\
    & & \det(D^2\Phi|_{\Int \partial \Delta}) = \mu_M \nonumber \\
    & & \overline {\partial\Phi(M_\mathbb R)} = \Delta^\vee \nonumber
\end{eqnarray}
and its dual
\begin{eqnarray}
    & & \Psi: N_{\mathbb R} \rightarrow \mathbb R \text{ convex} \label{eq:MA2} \\
    & & \det(D^2\Psi|_{\Int \partial \Delta^\vee}) = \nu_N \nonumber \\
    & & \overline {\partial\Psi(N_\mathbb R)} = \Delta. \nonumber
\end{eqnarray}
\begin{defn}
A reflexive polytope \(\Delta\) is \emph{optimal transport stable} if there exists an optimal transport plan from \(\mu_M\) to \(\nu_N\) which is supported on 
\begin{equation} 
\label{eq:Stability}
\bigcup_{m\in V(\Delta)} \Star(m)\times \tau_m \subset M_{\bbR}\times N_{\bbR} 
\end{equation}
\end{defn}

\begin{thm}
Assume that \(\Delta\) is a reflexive Weyl polytope, let \(W\) be the corresponding Weyl group and $m\in C^+_M\cap M$ be the lattice element used to generate $\Delta$. 
Then the optimal transport plan for \(W\)-invariant measures on \(\partial \Delta\) and \(\partial \Delta^{\vee}\) is supported on 
\begin{equation} 
\label{eq:OTPSupport}
\bigcup_{w \in W} w( \Star(m)\cap C_M^+ ) \times w^{\vee} (\tau_m\cap C_N^+) 
\end{equation}
In particular, \(\Delta\) is optimal transport stable, and the corresponding Monge-Ampère equations \eqref{eq:MA1} and \eqref{eq:MA2} admits solutions, which are smooth everywhere but on a piecewise affine co-dimension 2 set and a co-dimension 1 set respectively. 
\end{thm}

\begin{proof}
Let $\gamma$ be a $W$-invariant optimal transport plan (this always exist since the optimal transport problem is convex and $W$-invariant). We will prove that $\gamma$ is supported on \eqref{eq:OTPSupport}. Let $(m,n)$ be a point in the support of $\gamma$. By $W$-invariance, we may without loss of generality assume that $m\in C^+_M$. For any $\alpha\in S$ we have by Proposition~\ref{prop_OT_and_reflection} that the sign of $\langle m, \alpha \rangle$ is the same as the sign of $\langle \alpha^\vee,n \rangle$. Since $m\in C^+_M$ we get that $n\in C^+_N$, hence that $(m,n)$ lies in \eqref{eq:OTPSupport}. By Proposition~\ref{prop:Cones} it follows that $\gamma$ is supported on \eqref{eq:Stability}. It follows by \cite[Theorem 1]{Andreasson-Hultgren} that \eqref{eq:MA1} and \eqref{eq:MA2} admits solutions and if either $\Delta$ or $\Delta^\vee$ is Delzant, then the weak metric SYZ conjecture holds for the Dwork family in the corresponding toric Fano manifold. 

In order to get higher regularity, note that $\mu_M(C^+_M) = \nu_N(C^+_N)$. It follows that the optimal transport problem can be reduced to a collection of planar optimal transport problems as in \cite[Theorem~3]{Andreasson-Hultgren} and \cite{AHJMM_2024}. More precisely, let $U_m$ be the relative interior in $\partial\Delta$ of 
$$\bigcup_{w} wC^+_M$$
where the union is taken over all $w\in W$ such that $w C^+_N\subset \tau_m$. Then the restriction of $\gamma$ to $U_m\times \tau_m$ is (after scaling) an optimal transport plan from $\mu_M|_{U_m}$ to $\nu_N|_{\tau_m}$. This can be reduced to a planar optimal transport problem (useing the quotient maps to $d$-dimensional vector spaces $M_\mathbb R\rightarrow  M_\mathbb R/\mathbb R m$ and $N_\mathbb R\rightarrow  N_\mathbb R/\mathbb R n$ for some vertex $n\in \tau_m$). It follows from \cite{CaffarelliInt, CaffarelliBoundary} that $\Phi-n$ is smooth and globally $C^{1,\alpha}$ on $U_m$. By $W$-symmetry of $\Phi$ the same conclusion holds on 
$$ \bigcup_{w\in W} wU_m $$
and we get that the $c$-gradient $\partial^c \Phi:\partial \Delta \rightarrow \partial\Delta^\vee$ and $\partial^c \Psi:\partial \Delta^\vee \rightarrow \partial\Delta$ are Hölder continuous homeomorphisms, one the inverse of the other. From this, it follows that $\Phi$ and $\Psi$ are convex, and hence smooth, when restricted to open faces of $\partial \Delta$ and $\partial\Delta^\vee$, respectively. Consequently, the discriminant locus of $\Phi$ is given by the intersection of the boundary of $\bigcup_{w\in W} wU_m$ with the $d-1$ dimensional faces in $\Delta$: a piecewise affine set of dimension $d-2$. %Moreover, the discriminant locus of $\Psi$ is the image of this set under the Hölder continuous map $\partial^c \Phi$: a set of Hausdorff dimension $d-1-\alpha$, where $\alpha$ is the Hölder exponent of $\partial^c \Phi$ (\textcolor{blue}{This is probably not true, I need to check the Hölder exponents more carefully}.
The discriminant locus of $\Psi$ is the image of this set under the Hölder continuous map $\partial^c \Phi$. This image lies in the $d-1$-dimensional faces of $\Delta^\vee$. 
\end{proof}

\bibliographystyle{alpha}
\bibliography{SYZRS}

\end{document}